\documentclass[12pt,twoside]{amsart}
\usepackage{amsmath}
\usepackage{amssymb}
\usepackage{amscd}
\usepackage{xypic}
\xyoption{all}
\usepackage{todonotes}

\title[Uniqueness of factorization for fusion-invariant representations]{Uniqueness of factorization for fusion-invariant representations}

\author{Jos\'e Cantarero}
\thanks{}
\address{
\hfill\break Centro de Investigaci\'on en Matem\'aticas, A.C. Unidad M\'erida \\
\hfill\break Parque Cient\'ifico y Tecnol\'ogico de Yucat\'an \\ 
\hfill\break Carretera Sierra Papacal--Chuburn\'a Puerto Km 5.5 \\
\hfill\break Sierra Papacal, M\'erida, YUC 97302 \\
\hfill\break Mexico.}
\email{cantarero@cimat.mx}

\author{Germ\'an Combariza}
\thanks{}
\address{
\hfill\break Universidad Externado de Colombia  \\
\hfill\break Departamento de Matem\'aticas \\ 
\hfill\break Calle 12 No. 1-17 Este.  \\
\hfill\break C.P. 111711, Bogot\'a \\
\hfill\break Colombia.}
\email{german.combariza@uexternado.edu.co}

\newcommand{\comments}[1]{}

\newcommand{\Aut}{\operatorname{Aut}\nolimits}

\newcommand{\Conv}{\operatorname{Conv}\nolimits}
\newcommand{\diag}{\operatorname{diag}\nolimits}

\newcommand{\Hom}{\operatorname{Hom}\nolimits}

\newcommand{\Irr}{\operatorname{Irr}\nolimits}

\newcommand{\Out}{\operatorname{Out}\nolimits}
\newcommand{\reg}{\operatorname{reg}\nolimits}
\newcommand{\Rep}{\operatorname{Rep}\nolimits}
\newcommand{\rk}{\operatorname{rk}\nolimits}

\def \C{{\mathbb C}}
\def \F{{\mathbb F}}

\def \N{{\mathbb N}}

\def \Z{{\mathbb Z}}

\newcommand{\Ff}{{\mathcal{F}}}
\newcommand{\Ll}{{\mathcal{L}}}
\newcommand{\pcom}{^\wedge_p}

\newcommand{\higherlim}[2]{\displaystyle\setbox1=\hbox{\rm lim}
	\setbox2=\hbox to \wd1{\leftarrowfill} \ht2=0pt \dp2=-1pt
	\setbox3=\hbox{$\scriptstyle{#1}$}
	\def\test{#1}\ifx\test\empty
	\mathop{\mathop{\vtop{\baselineskip=5pt\box1\box2}}}\nolimits^{#2}
	\else
	\ifdim\wd1<\wd3
	\mathop{\hphantom{^{#2}}\vtop{\baselineskip=5pt\box1\box2}^{#2}}_{#1}
	\else
	\mathop{\mathop{\vtop{\baselineskip=5pt\box1\box2}}_{#1}}%
	\nolimits^{#2}
	\fi\fi}

\theoremstyle{plain}
\newtheorem*{introtheorem}{Theorem}
\newtheorem{theorem}{Theorem}[section]
\newtheorem{proposition}[theorem]{Proposition}
\newtheorem{corollary}[theorem]{Corollary}
\newtheorem{lemma}[theorem]{Lemma}
\newtheorem{conjecture}[theorem]{Conjecture}

\theoremstyle{definition}
\newtheorem{definition}[theorem]{Definition}
\newtheorem{remark}[theorem]{Remark}
\newtheorem{example}[theorem]{Example}



\begin{document}

\begin{abstract}
Given a saturated fusion system $\Ff$ over a finite $p$-group $S$, we provide
criteria to determine when uniqueness of factorization into irreducible $\Ff$--invariant
representations holds. We use them to prove uniqueness of factorization when the order of $S$
is at most $p^3$. We also describe an example where the monoid of fusion-invariant
representations is not even half-factorial. Finally, we find other examples of
fusion systems where this monoid is not factorial using GAP.
\end{abstract}

\maketitle

\section*{Introduction}

The $K$-theory of the classifying space of a finite group $G$ is determined
by the Atiyah-Segal completion theorem \cite{A}, \cite{AS} which established 
it to be the completion of the representation ring of $G$ with respect to the augmentation ideal. 
This is generally different from the Grothendieck group of complex vector
bundles over $BG$, which is described in \cite{JO}. 

These results were generalized recently to classifying spaces of $p$-local 
finite groups. Given a $p$-local finite group $(S,\Ff,\Ll)$, the Grothendieck
group of complex vector bundles over $| \Ll| \pcom$ was found in \cite{CCM} to
be isomorphic to the representation ring $R(\Ff)$, which can expressed as an
inverse limit of representation rings over the orbit category of $\Ff$. On the
other hand, the $K$-theory of $|\Ll|\pcom$ is isomorphic to the completion of
$R(\Ff)$ with respect to the augmentation ideal (see \cite{BC}).

The representation ring $R(\Ff)$ can also be regarded as the Grothendieck group
associated to the monoid of equivalence classes of representations of $S$ which 
are $\Ff$--invariant. For instance, when $\Ff$ is the fusion system associated
to a finite group $G$ with $p$-Sylow subgroup $S$, they are precisely the 
representations of $S$ whose characters are constant on $G$--conjugacy
classes of elements of $S$. We begin the paper by giving a brief review of 
fusion-invariant representations and their main properties in Section \ref{SectionReview}.
Even though we prove our results in the full generality of saturated fusion systems,
our main examples come from finite groups, hence this paper should be of interest
to people who are only interested in the representation theory of finite groups.

The group $R(\Ff)$ is free abelian and its rank is the number of $\Ff$--conjugacy 
classes of elements of $S$. However, there are examples in \cite{G} and \cite{Re}
where an $\Ff$--invariant representation does not decompose as sum of 
irreducible $\Ff$--invariant representations in a unique way. This represents
a significant departure from the behaviour of the representation ring for finite
groups, and this article aims to provide a better understanding of this phenomenon.

This lack of uniqueness is better understood from the point of view of the monoid $\Rep(\Ff)$
of isomorphism classes of $\Ff$--invariant representations. Using the language of
reduced commutative monoids, this is a submonoid of the free monoid $\Rep(S)$, and 
these examples show that it is not necessarily free. Free monoids are also called 
factorial to stress the uniqueness of factorization as a sum of atoms. 
We find several criteria in Section \ref{SectionMonoids} that guarantee the uniqueness of 
factorization, some of them characterizations in fact. The following theorem
outlines the main results in this section.

\begin{introtheorem}
Let $\Ff$ be a saturated fusion system over a finite $p$-group $S$. If one of the
following conditions hold, then $\Rep(\Ff)$ is factorial.
\begin{itemize}
\item The number of $\Ff$--conjugacy classes of elements of $S$ equals the number 
of irreducible $\Ff$--invariant representations of $S$.
\item As a free abelian group, $R(\Ff)$ has a basis $\{ \rho_1, \ldots, \rho_k \}$ of representations which satisfies 
\[ \Irr(\rho_j) - \bigcup_{i \neq j} \Irr(\rho_i) \neq \emptyset \]
for all $j$. 
\item As a free abelian group, $R(\Ff)$ has a basis of disjoint virtual representations.
\item The fusion system $\Ff$ is transitive.
\item The essential rank of $\Ff$ is zero.
\end{itemize}
Moreover, the first two conditions are characterizations of the factoriality
of $\Rep(\Ff)$.
\end{introtheorem}

The essential rank of any saturated fusion system over an abelian $p$-group $S$ is zero. 
In particular, factoriality is guaranteed when the order of $S$ is at most
$p^2$. It is natural to wonder if this holds when the order is $p^3$. Saturated 
fusion systems over groups of order eight are few and easy to handle, and over
groups of order $p^3$ with $p$ odd, there is a classification in \cite{RV}. 
Using a case-by-case analysis and the criteria from Section \ref{SectionMonoids}, 
we determine the factoriality of $\Rep(\Ff)$ for all these possibilities. The following 
theorem corresponds to Theorem \ref{smallorder}, which is the main result in Section \ref{SectionSmall}.

\begin{introtheorem}
If $|S| \leq p^3$ and $\Ff$ is a saturated fusion system over $S$, then $\Rep(\Ff)$ is a factorial monoid.
\end{introtheorem}

This result can not be taken further, since the examples without uniqueness of factorization
in \cite{G} and \cite{Re} correspond to fusion systems over groups of order $3^4$ and $2^4$, respectively.
They are in fact the fusion systems of $\Sigma_9$ and $PSL_3(\F_3)$ at the primes $3$ and $2$. One
may wonder if weaker factorization properties hold in $\Rep(\Ff)$ in general. One such property is 
half-factoriality, where we do not ask for uniqueness of factorization, but only for uniqueness of 
the length of factorization of each element. In Section \ref{SectionHalf}, we show that $\Rep(\Ff)$ 
is not half-factorial when $\Ff$ is the fusion system of $A_6$ at the prime $2$.

Even though the example of $A_6$ is completely described mathematically, we found this example thanks
to algorithms that we programmed and implemented in GAP \cite{GAP4}. Namely, for a fusion system $\Ff$ 
associated to a finite group $G$ with $p$-Sylow subgroup $S$, we built functions which
can determine the number of $\Ff$--conjugacy classes, the complexification of the representation ring, 
and the number of irreducible $\Ff$--invariant representations which are subrepresentations
of the regular representations. In some cases, we noticed that these pieces of information can be used 
to determine all the irreducible $\Ff$--invariant representations. Apart from the
examples from \cite{G} and \cite{Re}, we found new examples of fusion systems without uniqueness of
factorization into irreducible fusion-invariant representations. Namely, they are the fusion systems of
the groups:

\begin{itemize}
\item $\Sigma_6$, $\Sigma_8$, $PSL_2(\F_{17})$, $PSU_3(\F_5)$, $M_{10}$, $PSL_3(\F_5)$, $PSL_2(\F_{31})$, $PSU_3(\F_9)$ and $GL_3(\F_3)$ at the prime $2$.
\item $A_9$ and $PSL_3(\F_{19})$ at the prime $3$.
\end{itemize}

Apart from the non-factorial examples above, we found that factoriality holds for $A_8$ at the prime $2$ and 
for $PSp_4(\F_3)$ and $PSL_4(\F_7)$ at the prime $3$. Even though we have characterizations of the factoriality 
of $\Rep(\Ff)$, they are not completely satisfactory because they require determining the number of irreducible 
$\Ff$--invariant representations or studying the inclusion of $R(\Ff)$ into $R(S)$ to find a particular basis, 
but we do not have a direct criterion in terms of $\Ff$. 

There is another interesting aspect of the irreducible $\Ff$--invariant representations
which we could not settle completely in this paper. In Section \ref{SectionMonoids} we
show that if $\Rep(\Ff)$ is factorial, these representations must be subrepresentations
of the regular representation of $S$. We conjecture that this is always the case, even
when $\Rep(\Ff)$ is not factorial, and indeed it holds for all the examples where we
managed to determine all the irreducible $G$--invariant representations.

\section{Fusion-invariant representations}
\label{SectionReview}

In this section we provide an introduction to fusion-invariant
representations. Note that all the representations in this paper
are finite-dimensional unitary representations and we are only
interested in representations up to equivalence. We begin with a 
particular case which covers most of the examples found in this paper. 
However, the main results apply to the general definition.

\begin{definition}
\label{FusionInvariantConcrete}
Let $G$ be a finite group and $S$ a $p$-Sylow subgroup of $G$. We say that
a representation $\rho \colon S \to U_n$ is $G$--invariant
if the representations $ \rho_{|P}$ and $\rho_{|gPg^{-1}} \circ c_g$ are equivalent
for any $P \leq S$ and any $ g \in G$ that satisfies $gPg^{-1} \leq S$.
\end{definition}

The category of subgroups of $S$ with $G$--subconjugations as morphisms appears
as the indexing category in the Theorem of Cartan-Eilenberg on the $p$-local cohomology
of finite groups (see for instance Theorem III.10.3 in \cite{Bro}). This theorem can be improved to a homotopical decomposition
of $BG \pcom$ and C. Broto, R. Levi and B. Oliver \cite{BLO} created a suitable language and techniques
to study these decompositions and other spaces with similar properties. In this
paper the authors define the notion of saturated fusion system over a finite $p$-group
$S$. This is a category whose objects are the subgroups of $S$ and whose morphisms
are group homomorphisms which are required to satisfy some axioms. 

The interested reader can see Definitions 1.1 and 1.2 in \cite{BLO} for the formal 
definition, but the intuitive idea is that these morphisms behave like conjugations by elements 
of a bigger group. In fact, the main source of examples of saturated fusion systems comes 
from finite groups. If $G$ is a finite group and $S$ is a $p$-Sylow subgroup of $G$, there is a 
saturated fusion system called $\Ff_S(G)$ whose morphism sets are given by 
\[ \Hom_{\Ff_S(G)}(P,Q) = \Hom_G(P,Q) = \{ f \colon P \to Q \mid f = c_g \text{ for some } g \in G \} \]
But there are also exotic fusion systems, meaning that they are not of this form. At this moment, there are
many constructions of exotic fusion systems in the literature, although we will not encounter them in this paper.
We direct the interested reader to \cite{LO} and \cite{RV} for some examples. Next we extend Definition 
\ref{FusionInvariantConcrete} to this general setting.

\begin{definition}
Let $\Ff$ be a saturated fusion system over a finite $p$-group $S$. We say that a representation 
$\rho \colon S \to U_n$ is $\Ff$--invariant if the representations 
$\rho_{|P}$ and $\rho_{|f(P)} \circ f$ are equivalent for any
$P \leq S$ and any $f \in \Hom_{\Ff}(P,S)$.
\end{definition} 

In the particular case of $\Ff_S(G)$, this corresponds to $G$--invariant
representations. When $\Ff$ is clear, we also call them fusion-invariant
representations. These representations have been studied in \cite{BC}, 
\cite{CCM}, \cite{G} and \cite{Re}, but also in \cite{CC}, \cite{NS}, \cite{Z1}
and \cite{Z2} in a more general context. 

In order to determine whether a representation is fusion-invariant, we can use 
character theory and certain families of subgroups which generate the fusion 
system in a certain sense. For the following lemma, we say that two elements $x$, $y$ of $S$ are $\Ff$--conjugate 
if there exists a morphism $f$ in $\Ff$ such that $f(x)=y$. The definitions of 
centric, radical and essential subgroups can be found in Definition 1.6 of \cite{BLO},
Definition A.9 of \cite{BLO} and in the second paragraph of page 2 in \cite{H}, respectively.

\begin{lemma}
Let $\rho \colon S \to U_n$ be a representation. The following
conditions are equivalent.
\begin{enumerate}
\item $\rho$ is $\Ff$--invariant.
\item $\rho_{|P}$ and $\rho_{|P} \circ f$ are isomorphic for any $P \leq S$
and any $f \in \Aut_{\Ff}(P)$.
\item $\rho_{|P}$ and $\rho_{|P} \circ f$ are isomorphic for any $\Ff$--centric
radical subgroup $P$ of $S$ and any $f \in \Aut_{\Ff}(P)$.
\item $\rho_{|P}$ and $\rho_{|P} \circ f$ are isomorphic for any $\Ff$--essential 
subgroup $P$ of $S$ and any $f \in \Aut_{\Ff}(P)$ and for $P=S$ with any $f \in \Aut_{\Ff}(S)$.
\item $\chi_{\rho}(x) = \chi_{\rho}(y)$ for any pair of $\Ff$--conjugate elements
$x$, $y$.
\end{enumerate}
\end{lemma}

In the case of $\Ff_S(G)$, the last condition states that a representation $\rho$ is $G$--invariant if and
only $\chi_{\rho}(x) = \chi_{\rho}(y)$ for any pair of $G$--conjugate elements of 
$S$. We now enumerate some immediate properties of $\Ff$--invariant representations.

\begin{itemize}
\item If $\alpha$ is equivalent to $\beta$ and $\alpha$ is $\Ff$--invariant, so
is $\beta$.
\item If $\rho$ and $\alpha$ are $\Ff$--invariant, so are $\rho \oplus \alpha$
and $\rho \otimes \alpha$.
\item The trivial representation and the reduced regular representation are 
$\Ff$--invariant for any $\Ff$.
\item If $\rho$ and $\rho \oplus \alpha$ are $\Ff$--invariant, so is $\alpha$.
\item If $n\rho$ is $\Ff$--invariant for some $n \neq 0$, so is $\rho$.
\item If $\rho \oplus \alpha$ is isomorphic to $\rho \oplus \beta$, then $\alpha$
is isomorphic to $\beta$.
\end{itemize}

\begin{definition}
An $\Ff$--invariant representation is called irreducible if it can not be decomposed
as the direct sum of two nontrivial $\Ff$--invariant representations.
\end{definition}

It should be clear that any $\Ff$--invariant representation can be decomposed as a direct sum of irreducible 
$\Ff$--invariant representations. Note that an irreducible $\Ff$--invariant representation 
is not necessarily irreducible if we regard it as a representation of $S$. The following 
example illustrates this. 

\begin{example}
\label{SigmaTres}
Consider the $3$-Sylow subgroup $S=\{ 1, (1,2,3), (1,3,2) \}$ of $\Sigma_3$. The irreducible
representations of $S$ are the $1$-dimensional representations $\rho_j$ determined by
\[ \rho_j(1,2,3) = e^{2 \pi i j/3} \]
for $j=0,1,2$. Up to equivalence, the representations of $S$ have the form $ \rho = m_0 \rho_0
+ m_1 \rho_1 + m_2 \rho_2$ with $m_i \in \N$. Since $S$ does not have any nontrivial proper subgroup and
$\Aut_{\Sigma_3}(S) = \{ 1, c_{(1,2)} \}$, we have that $\rho$ is $\Sigma_3$--invariant
if and only if $ \rho \circ c_{(1,2)} $ is equivalent to $\rho$. But
\[ \rho \circ c_{(1,2)} = m_0 \rho_0 + m_1 \rho_2 + m_2 \rho_1 \]
and so this holds if and only if $m_1=m_2$. Hence any $\Sigma_3$--invariant
representation is a nonnegative integral combination of $\rho_0$ and $\rho_1 + \rho_2$. 
These are irreducible $\Sigma_3$--invariant representations, since they can not
be decomposed as the sum of two nontrivial $\Sigma_3$--invariant subrepresentations.
Note that $\rho_1 + \rho_2$ is not irreducible as a representation of $S$.
\end{example}

The isomorphism classes of $\Ff$--invariant representations of $S$ 
form a semiring with the operations of direct sum and tensor product. We denote it by 
$\Rep(\Ff)$ and when $\Ff = \Ff_S(G)$, we denote it by $\Rep_G(S)$ instead. By definition, 
this is a subsemiring of the semiring $\Rep(S)$. For the purpose of studying the uniqueness 
of decomposition as a sum of irreducibles, we can ignore their multiplicative structures and regard them simply as commutative monoids. 

\begin{remark}
We chose to write $\Rep_G(S)$ despite the fact that if corresponds to $\Ff_S(G)$ to emphasize 
that its elements are classes of representations of $S$.
\end{remark}

The theory of $\Ff$--invariant representations has some similarities with the
theory of supercharacters and superclasses \cite{DI}. For instance, $\Ff$ determines
a partition of $S$ given by $\Ff$--conjugacy classes, the set $\{1\}$ is one of the
classes and characters of $\Ff$--invariant representations are constant on each of 
these $G$--conjugacy classes. We can appreciate the difference in the case of $\Ff_{D_8}(\Sigma_4)$. The 
group $D_8$ is partitioned into four $\Sigma_4$--conjugacy classes and Example 
4.5.4 in \cite{G} shows that any representation of $D_8$ which is $\Sigma_4$--invariant is generated by $1$, $X+Z$,
$Y+Z$ and $XY$, where $X$, $Y$ and $XY$ are the one-dimensional irreducible
representations of $D_8$ and $Z$ is the two-dimensional irreducible representation. If this partition of $D_8$ 
were a part of a supercharacter theory, there would exist a supercharacter which would contain 
$2Z$ as a subrepresentation. But then the partition of the set of irreducible representations 
of $D_8$ would have less than four elements. 

On the other hand, it is easier to produce a supercharacter theory over a finite $p$-group
$S$ which does not come from a saturated fusion system over $S$. The two-part partition 
of $S$ where all the nontrivial elements are in the same class always corresponds to a supercharacter
theory. But if $S$ is a $p$-group whose nontrivial elements have at least two different orders,
the subsets forming this partition can not be $\Ff$--conjugacy classes.

\section{Monoids with unique factorizations}
\label{SectionMonoids}

In this section we recall some terminology and results from the theory of reduced commutative 
monoids in order to study uniqueness of factorization in the monoids of fusion-invariant 
representations introduced in the previous section. We follow the terminology from \cite{GH}. 
Since our monoids are commutative, we will denote by $0$ the identity element. In the rest of 
the paper, $\Ff$ will denote a saturated
fusion system over a $p$-group $S$.

\begin{definition}
A commutative monoid is reduced if $0$ is the only invertible element. 
\end{definition}

The commutative monoid $\Rep(\Ff)$ is reduced because it is a submonoid of $\Rep(S)$, which is
reduced.

\begin{definition}
Let $M$ be a reduced commutative monoid. A non-zero element $x$ of $M$ is called an atom if
$x=y+z$ implies $y=0$ or $z=0$. 
\end{definition}

\begin{definition}
An atomic monoid is a reduced commutative monoid which is generated by its set of atoms.
\end{definition}

The monoid $\Rep(\Ff)$ is atomic and its atoms are the irreducible $\Ff$--invariant representations of $S$.

\begin{definition}
An atomic monoid is called factorial if every nonzero element can be expressed as a finite sum of atoms 
in a unique way up to reordering.
\end{definition}

This is also called a unique factorization monoid. It is easy to see that a reduced commutative monoid is factorial
if and only if it is free, that is, isomorphic to a monoid of the form
\[ \N X = \bigoplus_{x \in X} \N  \]
for some set $X$. For a finite group $G$, the monoid $\Rep(G)$ is factorial because any representation of $G$ can be decomposed as a direct sum of irreducible 
representations in a unique way up to reordering. The atoms of $\Rep_{\Sigma_3}(S)$ are the 
representations $\rho_0$ and $\rho_1 + \rho_2$ described in Example \ref{SigmaTres} and this monoid is factorial since
\[ m \rho_0 + n(\rho_1 + \rho_2) = m'\rho_0 + n'(\rho_1 + \rho_2) \]
implies $m=m'$ and $n=n'$.

In order to describe a non-factorial example, we introduce the Grothendieck group of 
$\Rep(\Ff)$, which is denoted by $R(\Ff)$ and called the representation ring of $\Ff$. 
We also use the notation $R_G(S)$ when $\Ff=\Ff_S(G)$. In general, $R(\Ff)$ is a free 
abelian group whose rank is the number of $\Ff$--conjugacy classes of elements of $S$
(see Corollary 2.2 in \cite{BC}). However, the monoid $\Rep(\Ff)$ is not free in general
as the following example from \cite{G} shows.

\begin{example}
\label{EjemploBueno}
Let $S$ be a $3$-Sylow subgroup of $\Sigma_9$. The monoid $\Rep_{\Sigma_9}(S)$
is not factorial. The following example was first found in Example 4.5.5 of \cite{G}.
The representation ring $R_{\Sigma_9}(S)$ is the free abelian group generated by the 
following $\Sigma_9$--invariant virtual representations
\[ 1, a_1+a_2+X, b_0+c_0+Z-X, b_1+b_2+c_1+c_2+2Z, X+Y \]
where $a_i$, $b_0$, $c_i$, $X$, $Y$ and $Z$ are representations of $S$ which do
not share any irreducible representations of $S$. A more convenient basis is given
by
\[ 1, a_1+a_2+X, b_0+c_0+Z+a_1+a_2, b_1+b_2+c_1+c_2+2Z, X+Y \]
By Remark \ref{Criterio}, we can conclude that $1$, $ b_0+c_0+Z+a_1+a_2$, $b_1+b_2+c_1+c_2+2Z$ and $X+Y$ are irreducible
$\Ff$--invariant representations. On the other hand, using the same criterion with the
basis
\[ 1, a_1+a_2+X, b_0+c_0+Z+Y, b_1+b_2+c_1+c_2+2Z, X+Y \]
we can conclude that $a_1+a_2+X$ and $b_0+c_0+Z+Y$ are irreducible $\Ff$--invariant
representations. Then the expression
\[ (a_1+a_2+X)+(b_0+c_0+Z+Y) = (b_0+c_0+Z+a_1+a_2)+(X+Y) \]
shows that $\Rep_{\Sigma_9}(S)$ is not factorial.
\end{example}

Another known example of a non-factorial monoid of fusion-invariant representations, shown in Example A.2 of \cite{Re},
is $\Rep_{PGL_3(\F_3)}(SD_{16})$.

\begin{remark}
\label{IsomorphismFusion}
Let $\Ff_1$ and $\Ff_2$ be saturated fusion systems over $S_1$ and $S_2$, respectively. These fusion
systems are isomorphic if there is an isomorphism of groups $\alpha \colon S_1 \to S_2$ and an isomorphism
of categories $\alpha' \colon \Ff_1 \to \Ff_2$ given on objects by $\alpha'(P) = \alpha(P)$. It is easy
to show that the isomorphism $\alpha$ induces an isomorphism of monoids $\Rep(\Ff_2) \to \Rep(\Ff_1)$ 
and an isomorphism of rings $R(\Ff_2) \to R(\Ff_1)$.
\end{remark}

The following result is well known in the theory of non-unique factorizations,
but we include the proof of this particular case for completeness.

\begin{proposition}
Let $M$ be an atomic monoid such that its Grothendieck group $K(M)$
is free abelian of finite rank. Then $M$ is factorial if and 
only if the rank of $K(M)$ equals the cardinality of the set of atoms
of $M$.
\end{proposition}

\begin{proof}
Let $X$ be the set of atoms of $M$. If $M$ is factorial, then it is 
isomorphic to $\N X$ and therefore $K(M) \cong \Z X$.

Now assume that $M$ is not factorial. Since the set $X$ of atoms of $M$ 
generates $K(M)$ as an abelian group, there is an epimorphism $ \varphi \colon \Z X \to K(M)$
induced by the universal map from $M$ to $K(M)$. Since $M$ is not factorial, 
there is a relation $\sum n_x x = 0$ in $K(M)$ with at least one $n_x \neq 0$. 
Hence $\varphi$ is not injective and there is an exact sequence
\[ 0 \to \Z Y \to \Z X \to K(M) \to 0 \]
with $Y$ nonempty. Since $K(M)$ is free abelian of finite rank, if $|X|$ is finite we have
\[ |X| = \rk K(M) + |Y| \]
and so $\rk K(M) < |X|$. If $X$ is infinite, then certainly $\rk K(M)$ differs from $|X|$.
\end{proof}

In particular, the proof shows that if $M$ is an atomic monoid such that $K(M)$ is
free abelian of finite rank, the number of atoms in these atomic monoids is greater
or equal than the rank of $K(M)$. 

\begin{corollary}
\label{Bound}
The number of irreducible $\Ff$--invariant representation is greater or equal than the 
rank of $R(\Ff)$.
\end{corollary}

\begin{corollary}
\label{CharacterizationNumber}
The monoid $Rep(\Ff)$ is factorial if and only if the number of $\Ff$--conjugacy classes
of elements of $S$ equals the number of irreducible $\Ff$--invariant representations of $S$.
\end{corollary}

It is illustrative to see an example where these numbers differ.

\begin{example}
\label{AtomsOfSigmaNueve}
We return to $\Rep_{\Sigma_9}(S)$ from Example \ref{EjemploBueno}, which we already know
not to be factorial. The rank of $R_{\Sigma_9}(S)$ is five and it is freely generated as an
abelian group by the following $\Sigma_9$--invariant virtual representations of $S$.
\[ 1, P= a_1+a_2+X, Q = b_0+c_0+Z-X, R = b_1+b_2+c_1+c_2+2Z, S = X+Y \]
We replace this basis of $R_{\Sigma_9}(S)$ by
\[ 1, P, Q'=Q+S = b_0+c_0+Z+Y, R, S \]
which is a now a basis of $\Sigma_9$--invariant representations.

Let $\rho = a1+bP+cQ'+dR+eS$ be an irreducible $\Sigma_9$--invariant representation,
where the coefficients are integers. By the irreducibility, if all the coefficients 
are nonnegative, then $\rho$ must be one of these five representations, which we already
know to be irreducible by Example \ref{EjemploBueno}.

Next assume that at least one coefficient is negative. This coefficient could not be 
$b$, because then $a_1$ and $a_2$ would appear with negative coefficient. For the same reason,
it could not be $a$, $c$ or $d$. Therefore let us assume that $ e \leq -1$. Then we must have 
$e+b \geq 0$ and $e+c \geq 0$, since $e+b$ is the coefficient of $X$ in $\rho$
and $e+c$ is the coefficient of $Y$ in $\rho$. But then
\[ \rho = a1 + (-e)(P-S+Q') + (b+e)P + (c+e)Q' + dR \]
All the coefficients are now nonnegative and $P-S+Q' = a_1 + a_2 + b_0 + c_0 + Z$ is a representation.
So $\rho$ could also be $P-S+Q'$ which we already knew from Example \ref{EjemploBueno}
to be irreducible. We have proved that there are exactly six irreducible $\Sigma_9$--invariant representations. 
\end{example}

Given a representation $\rho$ of $S$, let us denote by $\Irr(\rho)$ the set of its irreducible 
subrepresentations.

\begin{proposition}
\label{CharacterizationFuerte}
The monoid $\Rep(\Ff)$ is factorial if and only if $R(\Ff)$ has a basis 
$\{ \rho_1, \ldots, \rho_k \}$ of representations which satisfies 
\[ \Irr(\rho_j) - \bigcup_{i \neq j} \Irr(\rho_i) \neq \emptyset \]
for all $j$. Moreover, if this holds, this basis is the set of irreducible
$\Ff$--invariant representations.
\end{proposition}

\begin{proof}
Given such a basis $\{ \rho_1, \ldots, \rho_k \}$, if $\rho$ is an irreducible
$\Ff$--invariant representation, we must have
\[ \rho = n_1 \rho_1 + \ldots + n_k \rho_k \]
for some $n_j \in \Z$. Since $\rho_j$ contains an irreducible subrepresentation
which is not shared with any $\rho_i$ for $i \neq j$, the coefficient $n_j$ must
be nonnegative. Since $\rho$ is irreducible, we must have that $\rho = \rho_j$
for some $j$ and therefore the number of irreducible $\Ff$--invariant representations
is at most $k$. By Corollary \ref{Bound}, this number is exactly $k$ and thus $\Rep(\Ff)$ 
is factorial by Corollary \ref{CharacterizationNumber}. We also obtain the last claim.

On the other hand, if $\Rep(\Ff)$ is factorial, by Corollary \ref{CharacterizationNumber},
the number of irreducible $\Ff$--invariant representations equal the rank of $R(\Ff)$, which
is finite. Let $\rho_1, \ldots, \rho_k$ be the irreducible $\Ff$--invariant representations. 
Given $\rho_j$, if each of its irreducible subrepresentations appeared as a subrepresentation 
of some $\rho_i$ with $i \neq j$, then $\rho_j$ would be a subrepresentation
of 
\[ \sum_{i \neq j} n_i \rho_i \]
for certain nonnegative integers $n_i$. And this would contradict the fact that $\Rep(\Ff)$ 
is factorial. Hence $\{ \rho_1, \ldots, \rho_k \}$ is a basis of $R(\Ff)$ which satisfies the desired property.
\end{proof}

\begin{remark}
\label{Criterio}
We can use a similar setup to prove an irreducibility criterion. Assume that $R(\Ff)$ has a basis 
$\{ \rho_1, \ldots, \rho_k \}$ of representations and $\rho_j$ satisfies
\[ \Irr(\rho_j) - \bigcup_{i \neq j} \Irr(\rho_i) \neq \emptyset \]
Let us show that $\rho_j$ is an irreducible $\Ff$--invariant representation.

Let $\alpha \in \Irr(\rho_j) - \bigcup_{i \neq j} \Irr(\rho_i)$ and let $\rho_j = \gamma + \epsilon$, 
where $\gamma$ and $\epsilon$ are $\Ff$--invariant representations. Since $\rho_j$ is the only representation in the basis
that contains $\alpha$ as a subrepresentation, we may assume that $\gamma$ contains $\alpha$ with the same multiplicity as
in $\rho_j$. When we write $\gamma$ as an integral lineal combination of the $\rho_i$, the
coefficient of $\rho_j$ must be one, hence $\gamma = \rho_j + \gamma'$. Then $ \gamma' + \epsilon = 0$ from where $\gamma'=0=\epsilon$
and so $\rho_j = \gamma$.
\end{remark}

We say that two virtual representations of $S$ are disjoint if their decompositions as integral linear combinations
of irreducible representations of $S$ do not share any irreducible representation. More generally, given a finite 
number of virtual representations of $S$, we say that they are disjoint if any two of them are disjoint.

\begin{proposition}
\label{CasoDisjuntas}
If $R(\Ff)$ has a basis of disjoint virtual representations of $S$, then $\Rep(\Ff)$ is
factorial.
\end{proposition}

\begin{proof}
We first modify this basis by replacing each element of the form $-\rho$, where $\rho$ is a representation, 
with $\rho$. Let $\alpha$ be a virtual representation in this new basis of $R(\Ff)$ which is not a representation. 
If $\beta$ is an $\Ff$--invariant representation of $S$ and we express it in this basis, the coefficient of 
$\alpha$ must be zero because $\beta$ is a representation. Therefore all the $\Ff$--invariant representations 
of $S$ are generated by the set $X$ of elements in this basis which are representations. In particular, $X$ is
the set of irreducible $\Ff$--invariant representations of $S$ by Remark \ref{Criterio}. This shows that the number of irreducible $\Ff$--invariant
representations is at most the rank of $R(\Ff)$, hence equal to the rank by Corollary \ref{Bound}. Therefore
$\Rep(\Ff)$ is factorial by Corollary \ref{CharacterizationNumber}.
\end{proof}

Proposition \ref{CasoDisjuntas} is not a characterization of the
factoriality of $\Rep(\Ff)$ as the following
example shows.

\begin{example}
\label{ExampleSigma4}
The representation ring of $\Ff_{D_8}(\Sigma_4)$ was determined in Example 4.5.4 of \cite{G}.
Let $\{1,X,Y,Z,XY\}$ be the set of irreducible representations of $D_8$, where $Z$ is the two-dimensional
irreducible representation. Then $R_{\Sigma_4}(D_8)$ is the free abelian group generated by $ 1$, 
$X+Z$, $Y+Z$ and $XY$. By Proposition \ref{CharacterizationFuerte},
the monoid $\Rep_{\Sigma_4}(D_8)$ is factorial and these are all the possible
irreducible $\Sigma_4$--invariant representations.

However, it is not possible to find a basis for $R_{\Sigma_4}(D_8)$ of disjoint virtual representations.
Assume there is a basis $ \{ v_1, v_2, v_3, v_4 \}$ of disjoint virtual representations. Let us express the basis
we already have in terms of this basis. We can assume $\pm X$ appears in the decomposition of $v_1$ and then the coefficient of $v_1$ in $Y+Z$, $1$ and $XY$
would be zero. The analogous conditions must hold, up to reordering, for the pairs $(Y,v_2)$, $(1,v_3)$ and $(XY,v_4)$.
In particular $ X+Z = \pm v_1$ and $Y+Z = \pm v_2$, which contradicts the disjointness of $v_1$ and $v_2$.
\end{example} 

However, Proposition \ref{CasoDisjuntas} is useful in two extreme cases. The following terminology
is taken from the Introduction of \cite{HKKS}.

\begin{definition}
A fusion system $\Ff$ is transitive if there are only two $\Ff$--conjugacy
classes of elements of $S$.
\end{definition}

\begin{lemma}
\label{transitive}
If $\Ff$ is transitive, then $\Rep(\Ff)$ is factorial.
\end{lemma}

\begin{proof}
In this proof we work with the equivalent description of $R(\Ff)$ as the
Grothendieck group of $\Ff$--invariant virtual characters.

The two conjugacy classes must be the class of $1$ and the class of all
nontrivial elements. Let $\chi \in R(\Ff)$ and let $ m = \chi(1)$. 
Since $\chi$ is $\Ff$--invariant, we must have $\chi(s) = n$ for all $s \neq 1$.
We know that $m$ is an integer. We will prove that $n$ is an integer as well. 
Note that
\[ (\chi,1) = \frac{1}{|S|} \sum _{s \in S} \overline{\chi(s)} = \frac{1}{|S|} ( m + \bar{n}(|S|-1)) =
 \frac{m-\bar{n}}{|S|} + \bar{n} \]
is an integer, or equivalently
\[ \frac{m-n}{|S|}+n \]
On the other hand, since $S$ is a $p$-group, it is nilpotent and so $[S,S] \neq S$. Therefore
its abelianization is not trivial and $S$ has a nontrivial $1$-dimensional representation $\rho$.
We denote its character by $\rho$ as well. Then
\[ (\chi,\rho) = \frac{1}{|S|}\left( m+ \bar{n} \sum_{s \neq 1} \rho(s) \right)
= \frac{1}{|S|}\left( m + \bar{n} \sum_{s \in S} \rho(s) - \bar{n} \rho(1) \right) = \frac{1}{|S|}(m-\bar{n}) \]
is also an integer, hence
\[ \frac{m-n}{|S|} \]
is an integer. Therefore $n$ is an integer and $|S|$ divides $m-n$. Let $m-n = a|S|$ with $a \in \Z$. Then
\[ \chi = (a+n)1 + a \widetilde{\reg_S} \]
where $\widetilde{\reg_S}$ is the reduced regular representation. Since the reduced regular 
representation does not contain the trivial representation as a subrepresentation, this shows
\[ R(\Ff) = \Z 1 \oplus \Z \widetilde{\reg_S} \]
Thus $R(\Ff)$ has a basis of disjoint characters and it is factorial by Proposition \ref{CasoDisjuntas}.
\end{proof}

This second extreme case corresponds to the case when every morphism is
a restriction of an automorphism of $S$.

\begin{lemma}
\label{EssentialRank}
If the essential rank of $\Ff$ is zero, then $\Rep(\Ff)$ is factorial.
\end{lemma}

\begin{proof}
If the essential rank of $\Ff$ is zero, then a representation $\rho$
of $S$ is $\Ff$--invariant if and only if $ \rho \circ f$ is equivalent
to $\rho$ for every $f \in \Aut_{\Ff}(S)$. Note that $\rho$ is irreducible
if and only if $\rho \circ f$ is irreducible. Therefore $\Aut_{\Ff}(S)$ 
acts on the set $\Irr(S)$ of irreducible representations of $S$. Then
\[ R(\Ff) = R(S)^{\Aut_{\Ff}(S)} = \Z[\Irr(S)]^{\Aut_{\Ff}(S)}  \cong \Z[\Irr(S)/{\Aut_{\Ff}(S)}] \]
This shows that $R(\Ff)$ has a basis of disjoint characters and therefore 
it is factorial by Proposition \ref{CasoDisjuntas}.
\end{proof}

If $S$ is abelian, the essential rank of $\Ff$ is automatically zero and
$p$-groups of order at most $p^2$ are abelian, so it is natural to next 
treat the case when $|S|=p^3$. We will do this in the next section. 

\begin{remark}
If $\Rep(\Ff)$ is factorial and $\rho$ is an irreducible $\Ff$--invariant representation,
then $\rho$ is a subrepresentation of $\reg_S^n$ for some $n$, hence $\reg_S^n = \rho \oplus \alpha$. 
Since $\reg_S$ is $\Ff$--invariant, so is $\alpha$ and so we can decompose it as a direct 
sum of irreducible $\Ff$--invariant representations
\[ \reg_S^n = \rho \oplus \bigoplus \rho_i^{k_i} \]
On the other hand, since $\reg_S$ is $\Ff$--invariant, it has a decomposition
\[ \reg_S = \bigoplus \omega_i^{r_i} \]
where the $\omega_i$ are irreducible $\Ff$--invariant representations. By the 
uniqueness of factorization, $\rho = \omega_j$ for some $j$, in particular
$\rho$ is a subrepresentation of $\reg_S$.
\end{remark}

The argument in the previous remark shows that any irreducible$\Ff$--invariant representation is a subrepresentation
of the regular representation when $\Rep(\Ff)$ is factorial. Apart from the examples from \cite{G} and
\cite{Re}, we enumerate in the appendix several examples of saturated fusion systems where $\Rep(\Ff)$ is not factorial,
found using the software GAP. We have checked that many of them still satisfy that any irreducible $\Ff$--invariant 
representation is a subrepresentation of the regular representation (we did not include these details in the paper). 
Hence it is reasonable to pose the following conjecture.

\begin{conjecture}
\label{Regular}
Every irreducible $\Ff$--invariant representation of $S$ is a subrepresentation
of the regular representation.
\end{conjecture}

\section{Fusion systems over $p$-groups of small order}
\label{SectionSmall}

In this section we determine that the monoids of 
fusion-invariant representations for saturated fusion systems 
over $p$-groups of order $p^3$ are all factorial.

We begin with $p=2$. Let $S$ be a group of order eight and
$\Ff$ a saturated fusion system over $S$. If $S$ is abelian,
then the essential rank of $\Ff$ is zero, hence $\Rep(\Ff)$
is factorial by Lemma \ref{EssentialRank}. It is well known
that, up to isomorphism, there are three saturated fusion systems over $D_8$, the fusion 
systems of $D_8$, $S_4$ and $A_6$, and two over $Q_8$, corresponding 
to the groups $Q_8$ and $SL_2(\F_3)$. For both $D_8$ and $Q_8$, let 
us denote by $X$, $Y$ and $XY$ the one-dimensional irreducible 
representations and by $Z$ the two-dimensional irreducible
representation. The monoid of fusion-invariant representations is factorial
in all these cases:

\begin{itemize}
\item When the group is $D_8$ or $Q_8$, because it is
the monoid of representations of a group.
\item For $\Sigma_4$, we proved it in Example \ref{ExampleSigma4}.
\item For $A_6$, there are three $A_6$--conjugacy classes of
elements of $D_8$. It is easy to check that the irreducible
$A_6$--invariant representations are $1$, $X+XY+Z$ and $Y+Z$.
\item For $SL_2(\F_3)$, there are three $SL_2(\F_3)$--conjugacy classes of
elements of $Q_8$. It is easy to check that the irreducible
$SL_2(\F_3)$--invariant representations are $1$, $X+Y+XY$ and $Z$.
\end{itemize}

For the rest of the section we assume that $p$ is an odd prime. We first
consider the extraspecial group $S = p_+^{1+2}$ with presentation
\[ \langle a, b, c \mid a^7 = 1, b^7 = 1, c^7 = 1, [a,b]=c, [a,c]=1, [b,c] = 1 \rangle \]
Recall that saturated fusion systems over $p_+^{1+2}$ are determined by the automorphism group of 
$p_+^{1+2}$, the elementary abelian groups of rank $2$ which are centric radical and their automorphisms groups 
(see Corollary 4.5 in \cite{RV}). Therefore a representation $\rho$ of $S$ is $\Ff$--invariant if and only if
the following two conditions are satisfied:
\begin{enumerate}
\item $f^*\rho \cong \rho$ for all $f \in \Aut_{\Ff}(S)$.
\item $h^*\rho\vert_V \cong \rho\vert_V$ for all $h \in \Aut_{\Ff}(V)$ for each
$\Ff$--centric radical subgroup $V$ which is elementary abelian of rank $2$.
\end{enumerate}
The first condition states that $\rho$ is a fixed point of the action
of $\Aut_{\Ff}(S)$ on $\Rep(S)$. As in the proof of Lemma \ref{EssentialRank},
the fixed points form a free monoid, a basis given by adding over the orbits 
of the irreducible representations of $S$ under the action of $\Aut_{\Ff}(S)$.

By Theorem 5.5.4 in \cite{Go}, the group $S$ has $p^2$ irreducible one-dimensional representations and 
$p-1$ irreducible representations of dimension $p$. The one-dimensional representations come from the quotient
\[ S \to S/Z(S) \cong \Z/p \times \Z/p \]
and are denoted by $x^iy^j$, just like the corresponding representations of $\Z/p \times \Z/p$. Let $\omega$ denote a primitive $p^2$-root of unity. 
The remaining irreducible representations $\phi_j$ for $j=1,\ldots,p-1$ are described by their values on $a$ and $b$. Namely,
\[ \phi_j(a) = \diag (\omega^{j(1+(p-1)p)},\ldots,\omega^{j(1+p)},\omega^j,1) \]
and $\phi_j(b)$ is the linear transformation that sends $e_1$ to $e_p$ and $e_n$ to $e_{n-1}$ if $n \geq 2$. Here $\{ e_1, \ldots, e_p \}$
denotes the standard basis of $\C^p$. Note that $\phi_j(c) = w^j I $.
 
To determine the orbits of the irreducible representations of $S$ under the action of $\Aut_{\Ff}(S)$, we can use the following lemma.

\begin{lemma}
Let $f$ be an automorphism of $S$ represented by the matrix $ M = \left( \begin{array}{cc} r & r' \\ s & s' \end{array} \right)$
under the isomorphism $ \Out(S) \cong GL_2(\F_p)$. Then $ f^*(x^i y^j) = x^{ir+js} y^{ir'+js'}$
and $f^*(\phi_j) = \phi_{j\det(M)}$, where all the indices are taken mod $p$.
\end{lemma}

Hence we will start determining the monoid $\Rep(S)^{\Out_{\Ff}(S)}$ 
and then impose the conditions given by the elementary abelian subgroups. 
The elementary abelian subgroups of rank two of $S$ are the subgroups $V_i$ 
generated by $c$ and $ab^i$ for $i=0, \ldots, p-1$, and the subgroup $V_p$ generated by $c$ and $b$.
These subgroups have a natural action of $GL_2(\F_p)$ by identifying each of them with 
$\F_p^2$ in such a way that the generating elements mentioned above form the standard basis. 
Note that the irreducible representations of an elementary abelian group
$V$ of rank two can be identified with elements of $\Hom_{\F_p}(V,\F_p)$
and the action of $\Aut(V) \cong GL_2(\F_p)$ under this identification 
is the dual action, which corresponds to multiplication by the transposed matrix. 

We now go through the all the possibilities in Tables 1.1 and 1.2 from \cite{RV} to determine 
the factoriality of the monoid of fusion-invariant representations in each case. 
\newline

\noindent \begin{bf}Case 1:\end{bf} $ G = S \rtimes W$, where $p$ does not divide the order of $W$. It follows
from Lemma \ref{EssentialRank}.
\newline

\noindent \begin{bf}Case 2:\end{bf} $ G = (\Z/p)^2 \rtimes (SL_2(\F_p) \rtimes \Z/r)$, where $r$ divides $p-1$.
We have $\Out_G(S) \cong \Z/(p-1) \times \Z/r$, and its elements are diagonal matrices. Note that the action of
$\Z/r$ partitions the set of representations of the form $y^j$ with $j \neq 0$ into $ d = (p-1)/r$ orbits.
Let $Y_i$, for $i=1,\ldots,d$, be the sum of the representations in each of these disjoint sets. A similar argument
applies for representations of the form $x^i y^j $ with $ij \neq 0$, obtaining representations $B_i$ for $i=1,\ldots,d$.
A basis for  $\Rep(S)^{\Out_G(S)}$ is given by
\[ \left\{ 1, X = \sum_{i=1}^{p-1} x^i, Y_1,\ldots,Y_d, B_1,\ldots,B_d, Z = \sum_{j=1}^{p-1} \phi_j \right\} \]
Now let $\rho$ be a $G$--invariant representation of $S$. Then
\[ \rho = q1+lX+\sum_{i=1}^d s_iY_i+\sum_{i=1}^d t_i B_i +uZ \]
The restriction of $\rho$ to the subgroup $V_0$ generated by $a$ and $c$ is
\[ (q+s_1r+\ldots+s_dr)1+l \left( \sum_{i=1}^{p-1} x^i \right) + \sum_{i=1}^d t_i r \left( \sum_{i=1}^{p-1} x^i \right) + u \left( \sum_{i=0,j=1}^{p-1} x^i y^j \right) \]
Since $\Aut_G(V_0) \cong SL_2(\F_p) \rtimes \Z/r$, the elements $x$ and $xy$ are related through the action of $SL_2(\F_p)$
and therefore $ u = l+t_1r+\ldots+t_d r$. Thus
\[ \rho = q1+l(X+Z)+\sum_{i=1}^d s_iY_i+\sum_{i=1}^d t_i (B_i+rZ) \]
and hence $\Rep_G(S)$ is factorial by Proposition \ref{CharacterizationFuerte}.
\newline

\noindent \begin{bf}Case 3:\end{bf} $ G = L_3(p)$, where $3$ does not divide $p-1$. In this case 
$\Out_{L_3(p)}(S) \cong \Z/(p-1) \times \Z/(p-1)$, and its elements are the diagonal 
matrices in $GL_2(\F_p)$. A basis for $\Rep(S)^{\Out_{L_3(p)}(S)}$ is given by
\[ \left\{ 1, X, Y = \sum_{i=1}^{p-1} y^i, M = \sum_{ij \neq 0} x^i y^j, Z  \right\} \]
Now let $\rho$ be an $L_3(p)$--invariant representation of $S$. Then
\[ \rho = q1+rX+sY+tM+uZ \]
The restriction of $\rho$ to the subgroup $V_0$ generated by $a$ and $c$ is
\[ q1+r \left( \sum_{i=1}^{p-1} x^i \right) + s(p-1)1 + t \left( \sum_{ij \neq 0} x^i \right) + u \left( \sum_{i=0,j=1}^{p-1} x^i y^j \right) \]
Since $\Aut_{L_3(p)}(V_0) \cong GL_2(\F_p)$, its action on $V_0 - \{ 1 \}$ is transitive
and so is its action on the set of non-trivial representations. Hence all the 
coefficients of these non-trivial representations must all be equal, that is, $ r+(p-1)t = u$.
By a symmetric argument with the subgroup $V_p$ generated by $b$ and $c$, we obtain
$ s+(p-1)t=u$, hence $ r=s$. Therefore
\[ \rho = q1 + r(X+Y+Z) + t(M+(p-1)Z) \]
By Proposition \ref{CharacterizationFuerte}, $\Rep_G(S)$ is factorial. In fact, it is 
the free monoid generated by $1$, $X+Y+Z$ and $M+(p-1)Z$.
\newline

\noindent \begin{bf}Case 4:\end{bf} $G = L_3(p):2$, where $3$ does not divide $p-1$.
Note that an $L_3(p)$--invariant representation $\rho$ of $S$ is $(L_3(p):2)$--invariant
if and only if $f^*\rho \cong \rho$, where $f$ is the automorphism of $S$ that permutes
$a$ and $b$. Since $f^*$ fixes $X$, $Y$, $M$ and $Z$, we obtain that $\rho$ is $L_3(p)$--invariant 
if and only if it is $(L_3(p):2)$--invariant. In particular, we see that $S$ has three 
$L_3(p)$--conjugacy classes.
\newline

\noindent \begin{bf}Case 5:\end{bf} $G=^2 F_4(2)'$ and $J_4$ at the prime $3$. The fusion 
systems of these groups contain the fusion system of $L_3(3):2$, hence they either have 
the same conjugacy classes as $L_3(3):2$ or they have only two conjugacy classes. Therefore 
the monoid $\Rep_G(S)$ is factorial for these two groups by Lemma \ref{transitive} or by
the previous case.
\newline

\noindent \begin{bf}Case 6:\end{bf} $G=Th$ at the prime $5$. In this case all the elementary 
abelian subgroups of rank two are radical and $Th$--conjugate and have $GL_2(\F_5)$ as their 
group of automorphisms. Given a nontrivial element $x$ in $S$, either $x = c^j$, in which case 
$x$ is conjugate to $c$ in $Th$, or the subgroup generated by $c$ and $x$ is one of the elementary 
abelian subgroups of rank two of $S$ and again $x$ is conjugate to $c$ in $Th$. Therefore the fusion system of $Th$ 
is transitive and so $\Rep_G(S)$ is factorial by Lemma \ref{transitive}.
\newline

\noindent \begin{bf}Case 7:\end{bf} $G=L_3(p)$ when $3$ divides $p-1$. In this case $\Out_{L_3(p)}(S) \cong \Z/(p-1) \times \Z/[(p-1)/3]$ 
and it is the subgroup of $\Out(S)$ generated by the matrices
\[ \left( \begin{array}{cc}
           \xi & 0 \\
           0 & \xi^{-2} \end{array} \right)  \qquad \qquad \left( \begin{array}{cc}
           \xi^{-2} & 0 \\
           0 & \xi \end{array} \right) \]
where $\xi$ is a generator of $\F_p^{\times}$. A more convenient set of generators is given by
\[ \left( \begin{array}{cc}
           \xi^3 & 0 \\
           0 & 1 \end{array} \right)  \qquad \qquad \left( \begin{array}{cc}
           \xi^{-2} & 0 \\
           0 & \xi \end{array} \right) \]
Let us compute a basis for $\Rep(S)^{\Out_{L_3(p)}(S)}$. The orbit of $x$ contains $x^{\xi^3}$, hence also $x^{\xi^3\xi^{-2}} = x^{\xi}$
and all the powers of $x$. The same happens for $y$. To find the orbit of an element
$x^i y^j$ with $ij \neq 0$, note that the action of $\Out_{L_3(p)}(S)$ on the set $\{ x^i y^j \mid ij \neq 0 \}$
is free, hence this set breaks into three orbits of size $(p-1)^2/3$. If $x^i y^j$ is related to $x^{i'} y^{j'}$,
then $ i' = i \xi^{3k-2n}$ and $j' = j \xi^n$, and so
\[ j'^{-1} i' = j^{-1} \xi^{-n} i \xi^{3k-2n} = j^{-1} i \xi^{3k-3n} \]
This condition breaks the set $\{ x^i y^j \mid ij \neq 0 \}$ into three subsets, those that satisfy $j^{-1} i = \xi^m$
with $m$ congruent to $0$, $1$ and $2$ mod $3$, respectively. Since each of these three subsets have size $(p-1)^2/3$,
these are the three orbits. Finally, since the determinant of the second matrix is $\xi^{-1}$, which is a generator of
$\F_p^{\times}$, the orbit of $\phi_1$ includes all the $\phi_j$. Hence a basis of $\Rep(S)^{\Out_{L_3(p)}(S)}$ is given by
\[ \left\{ 1, X, Y, Z, M_n = \sum_{j,k} x^{ja^{3k+n}} y^j \right\} \]
where $ 0 \leq n \leq 2$, the index $j$ runs over $\{1,\ldots,p-1\}$ and the index $k$ over $\{1,\ldots,(p-1)/3 \}$. Now let 
$\rho$ be an $L_3(p)$--invariant representation of $S$. Then
\[ \rho = q1+rX+sY+t_0M_0+t_1M_1+t_2M_2+uZ \]
The restriction of $\rho$ to the subgroup $V_0$ generated by $a$ and $c$ is
%
%
%
\[ q1+r \left( \sum_{i=1}^{p-1} x^i \right) + s(p-1)1 + \frac{(t_0+t_1+t_2)(p-1)}{3} \left( \sum_{i=1}^{p-1} x^i \right) + u \left( \sum_{i=0,j=1}^{p-1} x^i y^j \right) \]
Since $\Aut_{L_3(p)}(V_0) \cong SL_2(\F_p) \rtimes \Z/[(p-1)/3]$, the elements $x$ and $xy$
are related through the action of $\Aut_{L_3(p)}(V_0)$ and therefore the coefficients of 
$x$ and $xy$ must be equal. That is
\[ u = r + \frac{p-1}{3}(t_0+t_1+t_2) \]
By a similar argument with the subgroup $V_p$ generated by $b$ and $c$, we obtain
\[ u = s + \frac{p-1}{3}(t_0+t_1+t_2) \]
from where $r=s$. Therefore
\[ \rho = q1 + r(X+Y+Z) + t_0\left( M_0+ \frac{p-1}{3}Z \right) + t_1 \left( M_1+ \frac{p-1}{3}Z \right) + t_2\left( M_2+ \frac{p-1}{3}Z \right) \] 
By Proposition \ref{CharacterizationFuerte}, the monoid $\Rep_G(S)$ is factorial, it is in fact the 
free monoid generated by $1$, $X+Y+Z$ and $M_n+\frac{p-1}{3}Z$ for $n=0,1,2$.
\newline

\noindent \begin{bf}Case 8:\end{bf} $G=L_3(p):2$ when $3$ divides $p-1$. The
group $\Out_{L_3(p):2}(S)$ is generated by $\Out_{L_3(p)}$ and the class of 
the automorphism which permutes $a$ and $b$. Therefore a basis for 
$\Rep(S)^{\Out_{L_3(p):2}(S)}$. It is given by
\[ \{ 1, X + Y, M_0, M_1 + M_2, Z \} \] 
Considering the automorphisms of the subgroups $V_0$ and $V_p$ as above, we
obtain that $\Rep_G(S)$ is generated by $ 1, X+Y+Z, M_0 + \frac{p-1}{3}Z,
M_1 + M_2 + \frac{2(p-1)}{3}Z $, hence it is factorial
by Proposition \ref{CharacterizationFuerte}.
\newline

\noindent \begin{bf}Case 9:\end{bf} $G=L_3(p).3$ when $3$ divides $p-1$.
In this case, the group $\Out_{L_3(p).3}(S)$ is the subgroup
of diagonal matrices in $GL_2(\F_p)$ and therefore a basis for
$\Rep(S)^{\Out_{L_3(p).3}(S)}$ is given by $\{ 1, X, Y, M, Z \}$.
Considering the automorphisms of the subgroups $V_0$ and $V_p$, we obtain
that $\Rep_G(S)$ is generated by $ 1, X+Y+Z, M+6Z$, hence it is a factorial
monoid by Proposition \ref{CharacterizationFuerte}.
\newline

\noindent \begin{bf}Case 10:\end{bf} $G=L_3(p).\Sigma_3$ when $3$ divides $p-1$. Since the fusion system 
of $L_3(p).\Sigma_3$ contains the fusion system of $L_3(p).3$, it either has the same conjugacy 
classes as $L_3(p).3$ or two conjugacy classes. So $\Rep_G(S)$ is a factorial
monoid by the previous case or by Lemma \ref{transitive}.
\newline

\noindent \begin{bf}Case 11:\end{bf} $G=He$, the Held group, at the prime $7$. The group $\Out_{He}(S)$ 
is generated by the matrices
\[ \left( \begin{array}{cc}
 0 & 3 \\
 3 & 0 \end{array} \right), \left( \begin{array}{cc}
 2 & 0 \\
 0 & 1 \end{array} \right), \left( \begin{array}{cc}
 1 & 0 \\
 0 & 2 \end{array} \right) \]
Hence $\Rep(S)^{\Out_{He}(S)}$ is generated by the elements
\[ \{ 1,Z, A,B,C,D,E \} \]
where $A = x + x^2 + x^4 + y^3 + y^6 + y^5$ and $B = x^3 + x^6 + x^5 + y + y^2 + y^4$. The rest of representations
are given by
\begin{align*}
C & = xy + x^2 y + x^4 y + x y^2 + x^2 y^2 + x^4 y^2 + x y^4 + x^2 y^4 + x^4 y^4 + x^3 y^3 + x^6 y^3 + x^5 y^3  \\
  & + x^3 y^5 + x^6 y^5 + x^5 y^5 + x^3 y^6 + x^6 y^6 + x^5 y^6
\end{align*}
\[ D = x y^3 + x y^6 + x y^5 + x^2 y^3 + x^2 y^6 + x^2 y^5 + x^4 y^3 + x^4 y^6 + x^4 y^5 \]
\[ E = x^3 y + x^6 y + x^5 y + x^3 y^2 + x^6 y^2 + x^5 y^2 + x^3 y^4 + x^6 y^4 + x^5 y^4 \]
Now let $\rho$ be a $He$--invariant representation of $S$. Then
\[ \rho = q1+rA+sB+tC+uD+vE+wZ \]
The restriction of $\rho$ to the subgroup $V_1$ generated by $ab$ and $c$ is
\[ (q+3u+3v)1+(r+s+3t+u+v)\left( \sum_{i=1}^6 x^i \right)+ w \left( \sum_{i=0,j=1}^{p-1} x^i y^j \right) \]
Since $\Aut_{\Ff}(V_1)= SL_2(\F_7)$, the elements $x$ and $xy$ are related via
the dual action. Therefore $ w = r+s+3t+u+v$ and thus
\[ \rho = (q+3t+3u)1+r(A+Z)+s(B+Z)+t(C+3Z)+u(D+Z)+v(E+Z) \]
We obtain that $\Rep_G(S)$ is generated by $ 1, A+Z, B+Z, C+3Z, D+Z, E+Z$, hence it is factorial
by Proposition \ref{CharacterizationFuerte}.
\newline

\noindent \begin{bf}Case 12:\end{bf} $G=He:2$ at the prime $7$. A representation 
$\rho$ of $S$ is $(He:2)$--invariant if and only if it is $He$--invariant and $f^*\rho = \rho$, 
where $f$ is the order-two automorphism that permutes $a$ and $b$. Note that $f^*$ fixes $C$ and $Z$,
permutes $A$ and $B$, and permutes $D$ and $E$. Hence $\Rep_G(S)$ is generated by $1$, $A+B+2Z$, $C+3Z$ 
and $D+E+2Z$. By Proposition \ref{CharacterizationFuerte}, this is a factorial monoid.
\newline

\noindent \begin{bf}Case 13:\end{bf} $G=Fi'_{24}$, the derived Fischer group, at the prime $7$. The
fusion system of $Fi'_{24}$ contains the fusion system of $He:2$, but for the derived Fischer 
group, the subgroups $V_3$, $V_5$ and $V_6$ are radical with $SL_2(\F_7)$ as group of automorphisms.
Let $\rho = q1+r(A+B+2Z)+s(C+3Z)+t(D+E+2Z)$. The restriction of $\rho$ to the subgroup $V_6$ generated by
$ab^6$ and $c$ is
\[ (q+6s)1+r \left( 2 \sum_{i=1}^6 x^i + 2 \sum_{i=0,j=1}^{p-1} x^i y^j \right) + s \left( 2 \sum_{i=1}^6 x^i + 3 \sum_{i=0,j=1}^{p-1} x^i y^j \right)
+ t \left(3 \sum_{i=1}^6 x^i + 2 \sum_{i=0,j=1}^{p-1} x^i y^j \right) \]
Since $x$ and $xy$ are related through the dual action of $SL_2(\F_7)$, we must have
$ 2r+2s+3t = 2r+3s+2t$, that is, $t=s$ and therefore
\[ \rho = q1 + r(A+B+2Z) + s(C+D+E+5Z) \]
Hence $\Rep_G(S)$ is a factorial monoid by Proposition \ref{CharacterizationFuerte} and we see
that this fusion system has three conjugacy classes.
\newline

\noindent \begin{bf}Case 14:\end{bf} $G=Fi_{24}$, the Fisher group, at the prime $7$, and $\Ff = RV_1$.
The fusion system of the Fisher group $Fi_{24}$ and the exotic $7$-local
fusion system $RV_1$ both contain the fusion system of $Fi'_{24}$, so they 
either have the same conjugacy classes as $Fi'_{24}$ or two conjugacy
classes. So $\Rep_G(S)$ and $\Rep(\Ff)$ are factorial monoids by the
previous case or by Lemma \ref{transitive}.
\newline

\noindent \begin{bf}Case 15:\end{bf} $ G = O'N$, the O'Nan's group, at the prime $7$. In this case the 
group $\Out_{\Ff}(S)$ is generated by the matrices
\[ \left( \begin{array}{cc}
     3 & 0 \\
     0 & 3 \end{array} \right) \quad \left( \begin{array}{cc}
     1 & 0 \\
     0 & 6 \end{array} \right) \quad \left( \begin{array}{cc}
     0 & 1 \\
     1 & 0 \end{array} \right) \]
and therefore the monoid $\Rep(S)^{\Out_{\Ff}(S)}$ is generated by the representations
\[ \{ 1, X+Y, Z, L, N \}  \]
where 
\[ L = \sum_{i=1}^6 x^i y^i + \sum_{i=1}^6 x^i y^{7-i} \]
and $ N = M-L$. Let 
\[ \rho = m1+n(X+Y)+lL+rN+kZ \]
The restriction of $\rho$ to the subgroup $V_0$ generated by $a$ and $c$ is given by
\[ m1 + n1 + n \left( \sum_{i=1}^6 x^i \right) + 2l \left( \sum_{i=1}^6 x^i \right) + 4r \left( \sum_{i=1}^6 x^i \right) + k \left( \sum_{i=0, j=1}^6 x^i y^j \right) \]
Since $V_0$ is centric and radical, then $x$ and $xy$ are related through the action
of $SL_2(\F_7)$ and therefore $ n+2l+4r=k$. Similarly, the restriction to the subgroup
$V_1$ generated by $ab$ and $c$ is given by
\[ (m+6l)1 + 2n \left( \sum_{i=1}^6 x^i \right) + l \left( \sum_{i=1}^6 x^i \right) + 4r \left( \sum_{i=1}^6 x^i \right) + k \left( \sum_{i=0, j=1}^6 x^i y^j \right) \]
and so $ 2n+l+4r=k$. Then $n=l$ and $ k = 3n+4r$. Therefore $\Rep_G(S)$ is generated
by $ 1$, $X+Y+M+3Z$ and $N+4Z$ and so it is factorial by Proposition \ref{CharacterizationFuerte}. 
We also see that there are three conjugacy elements of elements of $S$. 
\newline

\noindent \begin{bf}Case 16:\end{bf} $G=O'N:2$ at the prime $7$, and $\Ff=RV_2, RV_3$. Since all these fusion
systems contain the fusion system of $O'N$, the corresponding monoids are factorial
by the previous case or Lemma \ref{transitive}.
\newline

\noindent \begin{bf}Case 17:\end{bf} $ G = M$, the Fischer-Griess Monster, at the prime $13$. Following \cite{Y},
the group $\Out_M(S) \cong 3 \times 4S_3$ is generated by the matrices
\[ r= \left( \begin{array}{cc}
              3 & 0 \\
              0 & 9 \end{array} \right) \quad s= \left( \begin{array}{cc}
              5 & -4 \\
              -2 & 7 \end{array} \right) \quad z= \left( \begin{array}{cc}
              2 & 2 \\
              1 & -2 \end{array} \right)  \]
The element $r$ has order three, while $z$ has order $24$ and satisfies
\[ z^2= \left( \begin{array}{cc}
              6 & 0 \\
              0 & 6 \end{array} \right) \]
Moreover, $s=zrz^{-1}$. Note that $6$ is a generator of $\F_{13}^{\times}$. We first compute the orbit of the 
irreducible representations of $S$ under the action of $\Out_G(S)$, which are
\[ \{ 1, Z, R, T \} \]
where
\begin{align*} 
R & = \sum_{n=1}^{12} (z^2)^n x + \sum_{n=1}^{12} (z^2)^n xy + \sum_{n=1}^{12} (z^2)^n xy^3 + \sum_{n=1}^{12} (z^2)^n xy^9 + \sum_{n=1}^{12} (z^2)^n xy^7
+ \sum_{n=1}^{12} (z^2)^n xy^8 \\
 & + \sum_{n=1}^{12} (z^2)^n xy^{11} + \sum_{n=1}^{12} (z^2)^n y 
\end{align*}
\[ T = \sum_{n=1}^{12} (z^2)^n xy^2 + \sum_{n=1}^{12} (z^2)^n xy^{10} + \sum_{n=1}^{12} (z^2)^n xy^6 + \sum_{n=1}^{12} (z^2)^n xy^4
+ \sum_{n=1}^{12} (z^2)^n xy^5 + \sum_{n=1}^{12} (z^2)^n xy^{12} \]
We obtained these orbits $R$ and $T$ as follows. It is clear that the action of $z^2$ is free on elements of the form $x^i y^j$.
On the other hand, the actions of $z$ and $r$ allows us to obtain from $x$ at least the following elements.
\[
\diagram
x \rto^z & x^2y^2 \dto^r        &                  &       \\
         & x^6y^5 \dto^r \rto^z & x^4y^2           &       \\
         & x^5y^6 \rto^z        & x^3y^{11} \dto^r &        \\
         &                      & x^9y^8 \rto^z    & y^{10} 
\enddiagram
\]
It is easy to see that these elements lie in the $z^2$-orbits of $xy$, $xy^3$, $xy^9$, $xy^7$, $xy^8$, $xy^{11}$ and $y$. 
Therefore the orbit of $x$ has at least $96$ elements. On the other hand, $1$, $rz^4 = \left( \begin{array}{cc}
                                                                                     1 & 0 \\
                                                                                     0 & 3 \end{array} \right) $ and $(rz^4)^2$ fix $x$,
hence the orbit has at most $96$ elements. This shows that $R$ is the orbit of $x$. To find the orbit of $xy^2$, note that
we can obtain the following elements using the actions of $z$ and $r$.
\[ 
\diagram
xy^2 \rto^z \dto^r & x^7y^5 \dto^r \\
x^3y^5 \dto^r            & x^8y^6 \dto^r \\            
x^9y^6 & x^{11}y^2
\enddiagram
\]
Again, it is easy to see that these elements lie in the $z^2$-orbits of $xy^{10}$, $xy^6$, $xy^4$, $xy^5$ and $xy^{12}$.
Therefore the orbit of $xy^2$ has at least $72$ elements. But there are only $168$ elements of the form $x^iy^j$ with $(i,j)\neq(0,0)$,
hence $T$ is the orbit of $xy^2$.

Now let $\rho= a1+nZ+mR+kT$ be a $G$--invariant representation of $S$. The restriction of $\rho$ to the subgroup $V_1$ generated by $ab$ and $c$ is
\[ (a+12k)1 + n \sum_{i=0,j=1}^12 x^i y^j + 8m \sum_{i=1}^{12} x^i + 5k \sum_{i=1}^{12} x^i \]
Since $V_1$ is radical and its group of automorphisms contains $SL_2(\F_{13})$, the coefficients of $xy$ and $x$ must be equal,
that is
\[ n = 6m+5k \]
Hence the irreducible $G$--invariant representations are $1$, $8Z+M$ and $5Z+N$. Therefore $\Rep_G(S)$
is a factorial monoid by Proposition \ref{CharacterizationFuerte}, or alternatively, by using the atlas 
of finite groups \cite{CCNPW} to check that there are two conjugacy classes of elements of order $13$ in $S$
and Corollary \ref{CharacterizationNumber}.

\begin{theorem}
\label{smallorder}
If $|S| \leq p^3$ and $\Ff$ is a saturated fusion system over $S$, then $\Rep(\Ff)$ is a factorial monoid.
\end{theorem}

\begin{proof}
If $|S| \leq p^2$, then $S$ is abelian and $\Ff$ has essential rank zero, 
hence $\Rep(\Ff)$ is factorial by Lemma \ref{EssentialRank}.

Thus we may assume that the order of $S$ is exactly $p^3$.
If $p=2$, then $S$ is abelian or isomorphic to $D_8$ or $Q_8$ and
we checked that $\Rep(\Ff)$ is factorial in those cases. If $p>2$,
then $S$ is either abelian or generalized extraspecial. If $S$ is 
isomorphic to $p^{1+2}_+$, we showed before the proposition that 
for any saturated fusion system $\Ff$ over $S$, the monoid $\Rep(\Ff)$ 
is factorial. In $S$ is not isomorphic to $p^{1+2}_+$, then the essential
rank of $\Ff$ is zero by Theorem 4.2 in \cite{St}, hence $\Rep(\Ff)$ is
factorial by Lemma \ref{EssentialRank}.
\end{proof}

Note that this result can not be extended to order $p^4$, as the non-factorial
examples in \cite{G} and \cite{Re} are given over Sylows of order $3^4$ and $2^4$,
respectively.

\section{The half-factoriality property}
\label{SectionHalf}

In this section we give an example of a monoid of fusion-invariant representations
for which uniqueness of factorization fails in such a way that not even the length
of an element is well defined. 

Let $M$ be an atomic monoid. Given a factorization of an element 
\[ x = \sum_{i=0}^k a_i \]
as a sum of atoms, we say that $k$ is the length of the factorization. Following the 
terminology in \cite{GH}, but keeping in mind that all our monoids are reduced, we
consider half-factorial monoids.

\begin{definition}
Let $M$ be an atomic monoid. We say that $M$ is half-factorial if for each element
$x \in M$, every factorization of $x$ has the same length. 
\end{definition}

Note that different elements may have factorizations of different lengths. The following 
proposition is inspired by Proposition 5.4 in \cite{Got}. Given an atomic monoid $M$, consider
the subset $\Conv(M)$ of elements of $K(M)$ which can be expressed in the form
\[ \sum n_i a_i \]
where the $a_i$ are atoms and $\sum n_i \neq 0$.

\begin{lemma}
\label{CharacterizationHalfFactorial}
Let $M$ be an atomic monoid. The following conditions are equivalent.
\begin{enumerate}
\item $M$ is half-factorial.
\item $ 0 \notin \Conv(M)$.
\item $\Conv(M) \neq K(M)$.
\end{enumerate}
\end{lemma}

\begin{proof}
Assume that $M$ is not half-factorial. Then there exists $x \in M$ that
can be expressed as a sum of atoms in two different ways
\[ x = \sum n_i a_i = \sum m_i b_i \]
with $\sum n_i \neq \sum m_i$. Then in $K(M)$ we have
\[ 0 = \sum n_i a_i - \sum m_i b_i \]
and so $0 \in \Conv(M)$.

Assume now that $ 0 \in \Conv(M)$ and let $x \in K(M)$. We must have $x = \sum n_i a_i$
where $a_i$ are atoms. If $ \sum n_i \neq 0$, then $x \in \Conv(M)$. If $\sum n_i = 0$,
then since $ 0 \in \Conv(M)$ we have $ 0 = \sum m_i b_i$ with $\sum m_i \neq 0$ and therefore
\[ x = x + 0 = \sum n_i a_i + \sum m_i b_i \]
from where $x \in \Conv(M)$.

Finally, assume $\Conv(M) = K(M)$. Since $ 0 \in \Conv(M)$, we have $ 0 = \sum n_i a_i$ with $\sum n_i \neq 0$.
Then
\[ \sum_{n_i > 0} n_i a_i = \sum_{n_i<0} (-n_i)a_i \]
and note that $\sum_{n_i>0} n_i$ can not equal $\sum_{n_i<0} (-n_i)$ since $\sum n_i \neq 0$. Hence $M$
is not half-factorial. 
\end{proof}

\begin{example}
We saw in Example \ref{EjemploBueno} that $\Rep_{\Sigma_9}(S)$ is not factorial when
$S$ is a $3$-Sylow subgroup of $\Sigma_9$. But we shall see that it is half-factorial. 
In Example \ref{AtomsOfSigmaNueve} we proved that the atoms of this monoid are given by 
$1$, $P$, $Q'$, $R$, $S$ and $P-S+Q'$. Let
\[ 0 = a1+bP+cQ'+dR+eS+f(P-S+Q') = a1+(b+f)P+(c+f)Q'+dR+(e-f)S \]
Since $R_{\Sigma_9}(S)$ is free abelian with basis $\{1,P,Q',R,S\}$, we
obtain $a=b+f=c+f=d=e-f=0$ and therefore 
\[ a+b+c+d+e+f = a+(b+f)+(c+f)+d+(e-f)=0 \]
Hence $ 0 \notin \Conv(\Rep_{\Sigma_9}(S))$ and so $\Rep_{\Sigma_9}(S)$ is
half-factorial by Proposition \ref{CharacterizationHalfFactorial}.
\end{example}

The following proposition is inspired by this example.

\begin{proposition}
\label{ConvexProposition}
Let $M$ be an atomic monoid such that $K(M)$ is free abelian of finite rank. 
If there exists a basis of $K(M)$ such that every atom of $M$ can be expressed 
as a convex integral linear combination of the elements of this basis, then $M$ is
half-factorial.
\end{proposition}

\begin{proof}
Let $\{b_1,\ldots,b_k\}$ be a basis of $K(M)$ fulfilling the condition
in the statement. For each atom $a$ of $M$, there exists a convex integral
linear combination
\[ a = \sum \lambda_j(a) b_j \]
Assume we have $0 = \sum n_a a$, where the sum runs over the set of atoms of $M$.
\[ 0 = \sum_a n_a a = \sum_a n_a \left( \sum_j \lambda_j(a) b_j \right) = \sum_j \left( \sum_a n_a \lambda_j(a) \right) b_j  \]
from where $\sum_a n_a \lambda_j(a) = 0 $ for each $j$. But then
\[ \sum_a n_a = \sum_a \left( \sum_j \lambda_j(a) \right) n_a = \sum_j \left( \sum_a n_a \lambda_j(a) \right) = 0 \]
This shows that $0 \notin \Conv(M)$, hence $M$ is half-factorial.
\end{proof}

Let us show that $\Rep_{\Sigma_6}(S)$ is not half-factorial when $S$ is a $2$-Sylow subgroup of $\Sigma_6$. We choose $S$ to be
the image of the composition $ D_8 \times \Z/2 \to \Sigma_4 \times \Z/2 \to \Sigma_6$, where the first map is induced by the
action of $\Sigma_4$ on the vertices of a square and the second map is induced by the standard inclusion of $\Sigma_4$ in $\Sigma_6$
and the element $(5,6)$. If $\beta$ is the sign representation of $\Z/2$, the irreducible representations of $S$ are of the
form $ \rho \otimes 1$ or $\rho \otimes \beta$, where $\rho \in \{ 1, X, Y , XY, Z \}$ is an irreducible representation of $D_8$.
We name the latter $\rho'_i$ according to the order in which they appeared in the previous line. We define for $ 1 \leq i \leq 10$
the representations 
\[ \rho_i = \left\{ \begin{array}{ll}
            \rho'_i \otimes 1 & \text{ if } i \leq 5 \\
            \rho'_{i-5} \otimes \beta & \text{ if } i > 5 \end{array} \right. \]
The group $ S $ has ten conjugacy classes and six $\Sigma_6$--conjugacy classes. Using the character criterion for a representation to be $\Sigma_6$--invariant, 
it is easy to show that 
\[ R_{\Sigma_6}(S) = \Z \{ \rho_1, \rho_2+\rho_4+\rho_5+\rho_7+\rho_{10},\rho_3+\rho_5+\rho_7+\rho_{10},-\rho_4+\rho_6-\rho_7-\rho_{10},\rho_4+\rho_8+\rho_{10}, \rho_9 \} \]
But a more convenient basis is given by
\[ R_{\Sigma_6}(S) = \Z \{ \rho_1, \rho_2+\rho_4+\rho_5+\rho_7+\rho_{10},\rho_3+\rho_5+\rho_7+\rho_{10},\rho_2+\rho_5+\rho_6,\rho_4+\rho_8+\rho_{10}, \rho_9 \} \]
Let us call these representations $\alpha_j$ for $j=1,\ldots,6$. Note that $\alpha_j$ with $j \neq 2$
are irreducible $\Sigma_6$--invariant representations by Remark \ref{Criterio}. Let $\alpha$ be an 
irreducible $\Sigma_6$--invariant representation of $S$. Then 
\[ \alpha = \sum n_i \alpha_i \]
If all the $n_i$ are nonnegative, then $\alpha$ must be one of the $\alpha_j$. On the other hand, the only coefficient that could be negative is $n_2 $ 
since each of the remaining $\alpha_i$ contains an irreducible representation of $S$ which is not shared with any other $\alpha_k$. Assume then $n_2 < 0$ 
and $n_i \geq 0$ for all $i \neq 2$. Then
\[ \alpha = n_1 \alpha_1 + (n_2+n_3) \alpha_3 + (n_2+n_4) \alpha_4 + (n_2+n_5) \alpha_5 + (-n_2)(\alpha_3+\alpha_4+\alpha_5-\alpha_2)+n_6 \alpha_6 \]
Note that 
\[ \alpha_3+\alpha_4+\alpha_5-\alpha_2 = \rho_3 + \rho_5 + \rho_6 + \rho_8 + \rho_{10} \]
Let $\alpha_7 = \rho_3 + \rho_5 + \rho_6 + \rho_8 + \rho_{10} $, which is $\Sigma_6$--invariant by the previous equation. Since $\alpha$ is
irreducible, only one coefficient must be one and the rest must be zero. Apart from the previous possibilities, we also obtain $\alpha=\alpha_7$.
Consider now the basis
\[ \{ \alpha_1, \alpha_2, \alpha_3, \alpha_5, \alpha_6, \alpha_7 \} \]
of $R_{\Sigma_6}(S)$. Note that $\alpha_2$ and $\alpha_7$ are the only elements of this basis which contain $\rho_2$ and $\rho_6$, respectively, 
hence they are irreducible by Remark \ref{Criterio}. Therefore $\Rep_{\Sigma_6}(S)$ is not factorial by Corollary \ref{CharacterizationNumber}. Moreover
\[ \alpha_3 + \alpha_4 + \alpha_5 = \alpha_2 + \alpha_7 \]
thus $\Rep_{\Sigma_6}(S)$ is not half-factorial. 

\appendix

\section{GAP algorithms for fusion-invariant representations}

We include here the algorithms related to fusion-invariant 
representations that we implemented in GAP for fusion systems 
of the form $\Ff_S(G)$, where $S$ is a $p$-Sylow subgroup of 
the finite group $G$.

\begin{itemize}
\item[(1)] The pattern of $G$--fusion of conjugacy classes of $S$.
\end{itemize} 

We take the list of $S$--conjugacy classes of elements of $S$ and check which of them
fuse to form a single $G$--conjugacy class of elements of $S$. The output is a list of 
positive integers with the same size as the list of $S$--conjugacy classes of $S$ such 
that the integers in two positions coincide if and only if the $S$--conjugacy classes
in those positions are $G$--conjugate. If $G$ is a symmetric group, we check $G$--conjugacy
using the cycle structure of elements.

It is easy to obtain from this process the number of $G$--conjugacy classes of elements
of $S$ and we add it to the output as an additional element in the list. 

\begin{itemize}
\item[(2)] Whether a representation of $S$ is $G$--invariant
\end{itemize}

Using the pattern of $G$--fusion of conjugacy classes of $S$, we check
for each two classes that were fused whether the characters coincide on
these classes. If at some point this is different from zero, we stop and 
return false. Otherwise it returns true. We have two versions of this
algorithm depending on whether the representation is given in terms of
the irreducible representations of $S$ or in terms of the $N_G(S)$--invariant
representations of $S$.

\begin{itemize}
\item[(3)] Whether $\Rep_G(S)$ is factorial, assuming Conjecture \ref{Regular} holds.
\end{itemize}

We compare the number of $G$--conjugacy classes of elements of $S$, which we can obtain from
the pattern of $G$--fusion of conjugacy classes of $S$, with the number of
irreducible $G$--invariant representations of $S$ which are subrepresentations
of the regular representations.

In the first version, we run over the list of subrepresentations of the regular representation of $S$. 
We build a list of irreducible $G$--invariant representations as follows. If the
representation is a subrepresentation of one in the list or a subrepresentation
of the complement of a representation in the list in the regular representation, we ignore it. Otherwise, we add
it to the list if it is $G$--invariant, or ignore it otherwise. If at some point
this list contains more elements that $G$--conjugacy classes of elements of $S$,
the function throws false, meaning that $\Rep_G(S)$ is not factorial. Otherwise,
it returns true, which would mean that $\Rep_G(S)$ is factorial if Conjecture
\ref{Regular} holds.

In the second version, we run over the list of those subrepresentations which
are $N_G(S)$--invariant, which must be determined previously. In the third version,
only for symmetric groups, we use the same list as in the second version, but using
the pattern of $G$--fusion of conjugacy classes that uses the cycle structure of
elements.

\begin{itemize}
\item[(4)] The list of irreducible $G$--invariant representations of $S$ which
are subrepresentations of the regular representation.
\end{itemize}

It follows the same algorithm as the previous function, but it does not stop until
it checks all the subrepresentations of the regular representation of $S$ in the
first version, and all the subrepresentations which are $N_G(S)$--invariant in the
second and third versions. It returns the list of irreducible $G$--invariant representations
of $S$ in terms of the irreducible representations of $S$.

\begin{enumerate}
\item[(5)] A basis for the complexification of the representation ring of $G$--invariant representations of $S$.
\end{enumerate}

Using the pattern of $G$--fusion of conjugacy classes of elements of $S$ (or the pattern
of cycle structure of elements if $G$ is a symmetric group), we find
pairs of elements of $S$ which determine all $G$--conjugacy relations between all
the elements of $S$. Then we create a matrix that has a row for each of these pairs.
The $j$th element of the row that corresponds to $(s_1,s_2)$ contains the element 
$\chi_j(s_1)-\chi_j(s_2)$, where $\chi_j$ is the character of the $j$th
irreducible representation of $S$ in the order established by GAP.

Finally we compute the integral nullspace of this matrix. We are regarding an element
$(a_1,\ldots,a_r)$ in the nullspace as the virtual character $a_1\chi_1 + \ldots + a_r \chi_r$.

\begin{remark}
Our algorithms to find whether $\Rep_G(S)$ and the list of irreducible $G$--invariant representations of $S$
may require a lot of computing time.. An alternative method to answer these two questions is to find the
complexification of the representation ring as a vector space in GAP using the algorithm described above, then
find the representation ring as an abelian group from this object, and finally follow the same process we used 
in Section \ref{SectionHalf} to find the irreducible $\Sigma_6$--invariant representations of a $2$-Sylow of $\Sigma_6$.
We illustrate this in the following example.
\end{remark}

\begin{example}
Let $S$ be a $3$-Sylow subgroup of $A_9$. Using GAP we find that a basis for $R_{A_9}(S) \otimes \C$ is
\begin{align*}
X  = \{ & \rho_1, -\rho_2-\rho_3-\rho_4+\rho_6-\rho_7+\rho_8,-\rho_2-\rho_3-\rho_4+\rho_5-\rho_7+\rho_9,\rho_2+\rho_3+\rho_{12}+\rho_{15}, \\
        & -\rho_2-\rho_3+\rho_{10}+\rho_{11}+\rho_{14}+\rho_{16},\rho_2+\rho_3+\rho_4+\rho_7+\rho_{13}+\rho_{17} \} 
\end{align*}
where $\rho_i$ are the irreducible representations in the order presented by GAP. We see that there are six $A_9$--conjugacy
classes of elements of $S$. It is easy to check that $R_{A_9}(S)=\Z X$. We can also find a basis of representations, namely
\begin{align*}
\{ & \rho_1, \rho_6+\rho_8+\rho_{13}+\rho_{17},\rho_5+\rho_9+\rho_{13}+\rho_{17},\rho_2+\rho_3+\rho_{12}+\rho_{15}, \\
   & \rho_4+\rho_7+\rho_{10}+\rho_{11}+\rho_{13}+\rho_{14}+\rho_{16}+\rho_{17}, \rho_2+\rho_3+\rho_4+\rho_7+\rho_{13}+\rho_{17} \}
\end{align*}
Let $\alpha_j$ be the $j$th element of this basis. By Remark \ref{Criterio}, the $A_9$--invariant
representations $\alpha_j$ with $ 1 \leq j \leq 5$ are irreducible. Let $\alpha$ be an irreducible
$A_9$--invariant representation of $S$ expressed as $\alpha = \sum n_i \alpha_i$ in $R_{A_9}(S)$.
The only $n_i$ that could be negative is $n_6$ and we can rewrite it as:
\[ \alpha = n_1\alpha_1 + n_2\alpha_2+n_3\alpha_3+(n_4+n_6)\alpha_4+(n_5+n_6)\alpha_5 + (-n_6)(-\alpha_6+\alpha_5+\alpha_4) \]
Let $\alpha_7 = -\alpha_6+\alpha_5+\alpha_4$. Note that
\[ \alpha_7 = \rho_{10}+\rho_{11}+\rho_{12}+\rho_{14}+\rho_{15}+\rho_{16} \]
Hence $\alpha_7$ is a new possibility for $\alpha$. Indeed $\alpha_7$ and $\alpha_6$ are both irreducible by Remark \ref{Criterio} since
$\{ \alpha_1,\alpha_2,\alpha_3,\alpha_5,\alpha_6,\alpha_7\}$ is also a basis of $R_{A_9}(S)$. We have found that there are seven irreducible
$A_9$--invariant representations of $S$, hence $\Rep_{A_9}(S)$ is not factorial by Corollary \ref{CharacterizationNumber}.
\end{example}

Using the algorithms above, we found new examples of non-factorial monoids of the form 
$\Rep_G(S)$, which we present in the following table together with the examples from \cite{G}
and \cite{Re}. We applied them to groups with Sylows of order $2^4$, $2^5$
and $3^4$ corresponding to the fusion systems in Propositions 4.3 and 4.4 of \cite{AOV}, to some of the
fusion systems in Tables 1 and 2 of Chapter 7 of \cite{MM}, and other classical groups. The column 
``Irr. $G$--invariant reps.'' refers to irreducible $G$--invariant representations of $S$ which are 
subrepresentations of the regular representation. The group $(\Z/3 \times \Z/9) \rtimes \Z/3$ in the
table is the group from the SmallGroups library of GAP with id $[81,9]$, which is the semidirect product 
coming from the action of $\Z/3$ determined by the matrix $\left( \begin{array}{cc} 1 & 0 \\
0 & 1 \end{array} \right)$.

\begin{center}
  \begin{tabular}{ | l | l | l | l | l | }
    \hline
    Group $G$ & Prime $p$ & $p$-Sylow $S$ & $G$--conj. classes of $S$ & Irr. $G$--invariant reps. \\ \hline
    $\Sigma_6$ & $2$ & $\Z/2 \times D_8$ & $6$ & $7$ \\ \hline
    $PSL_2(\F_{17})$ & $2$ & $D_{16}$ & $5$ & $7$ \\ \hline
    $PSU_3(\F_5)$ & $2$ & $SD_{16}$ & $5$ & $8$ \\ \hline
    $M_{10}$ & $2$ & $SD_{16}$ & $6$ & $8$ \\ \hline
    $SL_3(\F_3)$ & $2$ & $SD_{16}$ & $5$ & $8$ \\ \hline
    $PSL_3(\F_5)$ & $2$ & $\Z/4 \wr \Z/2$ & $7$ & $8$ \\ \hline 
    $PSL_2(\F_{31})$ & $2$ & $D_{32}$ & $9$ & $21$ \\ \hline
    $PSU_3(\F_9)$ & $2$ & $SD_{32}$ & $9$ & $53$ \\ \hline
    $GL_3(\F_3)$ & $2$ & $\Z/2 \times SD_{16}$ & $10$ & $16$ \\ \hline
    $\Sigma_8$ & $2$ & $D_8 \wr \Z/2$ & $10$ & ? \\ \hline    
    $\Sigma_9$ & $3$ & $\Z/3 \wr \Z/3$ & $5$ & $6$ \\ \hline
    $A_9$ & $3$ & $\Z/3 \wr \Z/3$ & $6$ & $7$ \\ \hline
    $PSL_3(\F_{19})$ & $3$ & $(\Z/3 \times \Z/9) \rtimes \Z/3$ & $7$ & $16$ \\ \hline
  \end{tabular}
\end{center}

There are other examples of finite groups where $\Rep_G(S)$ is not factorial aside from these. For instance, if $\Rep_G(S)$ is
not factorial and $H$ is any other finite group with $p$-Sylow subgroup $R$, then $\Rep_{G \times H}(S \times R)$ is not factorial.
Another instance comes from finite groups with the same fusion systems (by Remark \ref{IsomorphismFusion}), for example the fusion systems 
of $SL_3(\F_3)$ and $M_{11}$ at the prime $2$ are isomorphic \cite{MP}, hence $\Rep_{M_{11}}(S)$ is not factorial when $S$ is a
$2$-Sylow subgroup of $M_{11}$.

We also found using these functions that there is uniqueness of factorization for $PSp_4(\F_3)$ and $PSL_4(\F_7)$ at the prime $3$ and for $A_8$ at the
prime $2$. 

\bibliographystyle{amsplain}

\bibliography{mybibfile}

\end{document}